%% file: main.tex
\documentclass[fontsize=11pt,a4paper,DIV=12,numbers=enddot]{scrartcl}

\usepackage[T1]{fontenc}
\usepackage{libertine}

\usepackage{amsmath, amsthm, amsfonts, amssymb, mathtools, thm-restate, enumerate}

\usepackage{xcolor, graphicx, tikz}

\definecolor{oiBlack}{HTML}{000000}
\definecolor{oiOrange}{HTML}{E69F00}
\definecolor{oiSkyBlue}{HTML}{56B4E9}
\definecolor{oiGreen}{HTML}{009E73}
\definecolor{oiYellow}{HTML}{F0E442}
\definecolor{oiBlue}{HTML}{0072B2}
\definecolor{oiVermilion}{HTML}{D55E00}
\definecolor{oiPurple}{HTML}{CC79A7}
\colorlet{Black}{oiBlack}
\colorlet{Orange}{oiOrange}
\colorlet{SkyBlue}{oiSkyBlue}
\colorlet{Green}{oiGreen}
\colorlet{Yellow}{oiYellow}
\colorlet{Blue}{oiBlue}
\colorlet{Vermilion}{oiVermilion}
\colorlet{Purple}{oiPurple}
\usetikzlibrary{backgrounds,3d,intersections,calc,arrows.meta,fit}

\usepackage[section]{placeins}
\usepackage{etoolbox}
\pretocmd{\subsection}{\FloatBarrier}{}{}

\usepackage[subrefformat=simple,labelformat=simple]{subcaption}

\usepackage{xspace, xstring}
\usepackage{needspace}
\usepackage{booktabs,tabularx}
\newcounter{facefamily}
\usepackage[
  backend=biber,
  style=alphabetic,
  sortcites=true,
  sorting=ynt,
  maxbibnames=99,
  maxcitenames=99,
  minalphanames=3,
  giveninits=true,
  backref=true,
  doi=true,
  isbn=false,
  url=false,
  eprint=true
]{biblatex}

\usepackage{hyperref}
\hypersetup{colorlinks=true, citecolor=Blue, linkcolor=Blue, urlcolor=Blue}
\usepackage[noabbrev,capitalise]{cleveref}

\newcommand{\defn}[1]{\textsl{\textcolor{Blue}{#1}}}

\newtheorem{theorem}{Theorem}

\newtheorem{remark}[theorem]{Remark}

\newtheorem{lemma}[theorem]{Lemma}

\makeatletter
\def\@fnsymbol#1{%
   \ifcase#1
   \or L
   \or X
   \else
   \@ctrerr \fi
}%
\makeatother

\DeclareTOCStyleEntry[
    indent=0em,
    pagenumberformat=\normalfont,
    beforeskip=1pt plus .2pt,
    entryformat=\normalfont
    ]{tocline}{section}

\let\oldnewsavebox\newsavebox
\renewcommand{\newsavebox}[1]{%
  \ifdefined#1%
  \else
    \oldnewsavebox{#1}%
  \fi
}

\newcommand{\Lcal}{\mathcal L}

\tikzset{
 cover/.style={draw=black!55,line width=.45pt},
 matched/.style={draw=Orange,line width=1.4pt},
 completed/.style={draw=Blue,line width=1.4pt},
 facepoint/.style={circle,fill=black,inner sep=1.4pt},
 facelabel/.style={font=\small,fill=white,inner sep=1.2pt},
 rankline/.style={Orange!35,line width=.35pt},
 cubebox/.style={draw=black!50,dashed,rounded corners=3pt,fill=black!2,line width=.45pt}
}
\hypersetup{pdftitle={The Face Lattice of any Simplicial Polytope is Hamiltonian},
 pdfauthor={Robert Lauff}}

\begin{document}
\title{\huge The Face Lattice of any Simplicial Polytope is Hamiltonian}
\author{Robert Lauff\,\thanks{Technische Universit\"at Berlin, Germany, lauff@math.tu-berlin.de}}
\date{}
\maketitle
\vspace{-30pt}
\begin{abstract}
	\noindent{\textbf{Abstract.}}
	We show that the cover graph of the face lattice of every simplicial polytope admits a Hamilton cycle. Moreover, any two distinct facets can be prescribed as the neighbors of the full polytope on this cycle. By duality, the face lattice of every simple polytope is Hamiltonian as well. This proves the simplicial and simple cases of a conjecture of Behrooznia, Brenner, Merino, M\"utze, Rieck, and Verciani (SODA26).

	The proof uses a line shelling in both directions. The forward shelling partitions the faces into Boolean intervals, whose
	cover graphs are cubes. The reverse shelling gives a second partition, on which we construct a perfect matching by adding or
	removing a fixed vertex in each interval. This matching connects the cubes of the first partition. A theorem of Gregor then
	completes it to a Hamilton cycle using only edges within these cubes. The same argument applies to shellable simplicial
	spheres, although without the prescribed neighbors of the full complex in this case. Using Chari's shelling
	construction of perfect matchings, we also obtain Hamiltonicity for polytopes with at most one facet that is not a simplex.
\end{abstract}
\thispagestyle{empty}
\setcounter{page}{1}
\pagebreak

\section{Introduction}\label{sec:intro}

A listing of combinatorial objects in which consecutive objects differ by a small change is called a \defn{Gray code};
see~\cite{survey_gray_codes} for a survey. For subsets of a finite set, the natural change is to add or remove one element. The
corresponding graph is a cube, and, for a set of at least two elements, the binary reflected Gray code~\cite{GrayCode,KnuthBRGC}
gives a Hamilton cycle in this graph. When the objects are faces of a polytope, the analogous change is to pass from a face to a
containing face of dimension one higher, or to a contained face of dimension one lower. Thus, we seek a Hamilton cycle in the
cover graph of the face lattice.

Behrooznia, Brenner, Merino, M\"utze, Rieck, and Verciani~\cite{Listingfaces} conjectured that the face lattice of every polytope
is Hamiltonian. They provide evidence for this conjecture by proving it for a number of special cases, as well as some more
general families of polytopes, namely hypercubes, permutahedra and $B$-permutahedra, associahedra, cyclic polytopes,
three-dimensional polytopes, graph associahedra of chordal graphs, and quotientopes. Furthermore, the conjecture is supported by
the following observation: Every cover relation changes the parity of the dimension, so the cover graph is bipartite. Euler's
relation~\cite{ziegler1995polytopes} ensures that its two parts have the same size when the empty face and the full polytope are
included, as required for a Hamilton cycle. Hence, proving the conjecture would give an (overly complicated, but beautiful) proof
of Euler's formula. Note that Euler's formula already follows from Chari's theorem, as it implies that any face lattice admits a
perfect matching.

In this paper, we prove the Hamiltonicity of the face lattice for any simple or simplicial polytope. The facets of a simplicial
polytope are simplices, so the faces contained in any one facet form a Boolean lattice. The first idea is to use a
\defn{shelling}, an ordering in which the facets are added one by one, to assign each face to the first facet that contains it.
The shelling condition ensures that the faces assigned to a facet still form a Boolean interval. We therefore obtain a partition
of the face lattice into cubes.

Consider a tetrahedron with vertices $1,2,3,4$, and order its facets as
\[
	F_1=123,\qquad F_2=124,\qquad F_3=134,\qquad F_4=234,
\]
where we abbreviate vertex sets by strings. The first facet contributes all subsets of $123$. The second contributes the faces
  $4,14,24,124$, which form a square in the cover graph. The third contributes only $34$ and $134$, which form an edge. The last
  facet contributes just itself. Pairing this last facet with the full tetrahedron gives another edge, so all faces have been
  partitioned into cubes of positive dimension. See \Cref{fig:shelling}.

\begin{figure}[htbp]
	\centering
	\input{figures/shelling}
	\caption{A shelling of the tetrahedron, drawn as a Schlegel diagram with outer facet $123$. The facets are added in the order $123,124,134,234$. Blue indicates the new facet and orange its intersection with the preceding facets. The restriction faces, marked in purple, are $\emptyset,4,34,234$, respectively. A face of the new facet is new precisely when it contains the restriction face.}
	\label{fig:shelling}
\end{figure}

To connect these cubes, we construct a perfect matching. A theorem of Gregor~\cite{gregor2009perfect} then supplies a second
perfect matching, using only edges within the cubes, such that the two matchings together form a Hamilton cycle.

We obtain our matching from the same shelling that partitioned the face lattice, read backwards. A line shelling is reversible,
meaning that its reverse is also a shelling. So the reverse ordering gives a second partition into cubes. Matching along a
suitable coordinate in each reverse cube ensures that every cube of the forward partition, except the first, is connected to an
earlier cube, which implies connectivity.

The argument also allows one facet that is not a
	simplex. Placing this facet last in a line shelling
	preserves the forward partition into cubes. In the
	reverse partition, we use a perfect matching of its face
	lattice in place of a coordinate matching. Such matchings
	exist for every polytope by Chari's shelling
	construction~\cite[Proposition~4.1]{chari2000discrete}. We
	include this construction in the language of perfect
	matchings to explain its relation to our proof.

The paper is structured as follows. We state the results in \Cref{sec:results}, give the definitions and the required properties
of shellings in \Cref{sec:prelim}, and discuss related work in \Cref{sec:related}. In \Cref{sec:proofs}, we prove the main
theorem, and extend the argument to simplicial spheres and to polytopes with one exceptional facet. We also present Chari's
shelling proof of the perfect-matching theorem. We conclude with some open problems in \Cref{sec:conclusion}.

\begin{table}[p]
	\noindent
	Many interesting combinatorial
	objects arise as a face lattice. The
	face lattice of a simplex is the Boolean lattice, with a
	cube as its cover graph; permutahedra give ordered set
	partitions, and classical associahedra give partial
	triangulations. \Cref{tab:face-models} collects these and a large list of further examples:
	\par\medskip
	\centering
	\fontsize{9}{10}\selectfont
	\setlength{\tabcolsep}{4pt}
	\renewcommand{\arraystretch}{1}
	\setcounter{facefamily}{0}
	\newcommand{\family}[1]{\refstepcounter{facefamily}\label{fam:#1}\thefacefamily}
	\newcommand{\generalizes}[1]{\par{\fontsize{8}{9}\selectfont\itshape generalizes #1}}
	\begin{tabularx}{\textwidth}{@{}r >{\raggedright\arraybackslash}p{.36\textwidth} >{\raggedright\arraybackslash}X@{}}
		\toprule
		                             & Polytope                                                                                                                    & Face lattice \\
		\midrule
		\family{simplex}             & Simplex~\cite{ziegler1995polytopes}
		                             & Nonempty subsets of $[d+1]$, ordered by inclusion.                                                                                         \\
		\addlinespace[1pt]
		\family{cube}                & Cube~\cite{ziegler1995polytopes}
		                             & Strings in $\{0,1,*\}^d$; replacing a fixed bit by $*$ moves upwards.                                                                      \\
		\addlinespace[1pt]
		\family{cyclic}              & Cyclic polytope~\cite{ziegler1995polytopes}
		                             & Nonempty subsets of Gale-even $d$-subsets of $[n]$, ordered by inclusion; adjoin $\hat 1$.                                                 \\
		\addlinespace[1pt]
		\family{permutahedron}       & Permutahedron~\cite{ziegler1995polytopes}
		                             & Ordered partitions of $[d+1]$; merging consecutive blocks moves upwards.                                                                   \\
		\addlinespace[1pt]
		\family{bpermutahedron}      & $B$-permutahedron~\cite{booleStott1910semiregular,schoute1913polytopes}
		                             & Signed ordered partitions of $[d]$ with a zero block; merge adjacent signed blocks or absorb the first into the zero block.                \\
		\addlinespace[1pt]
		\family{associahedron}       & Classical associahedron~\cite{Associahedron}
		                             & Partial triangulations of a convex $(d+3)$-gon, ordered by reverse inclusion of diagonals.                                                 \\
		\addlinespace[1pt]
		\family{graph}               & Graph associahedron~\cite{Graph_associahedra_general}
		\generalizes{\ref{fam:permutahedron}, \ref{fam:associahedron}}
		                             & Tubings of a connected graph on $d+1$ vertices, ordered by reverse inclusion.                                                              \\
		\addlinespace[1pt]
		\family{nestohedron}         & Nestohedron~\cite{feichtner2005matroid,postnikov2009permutohedra}
		\generalizes{\ref{fam:simplex}, \ref{fam:cube}, \ref{fam:graph}}
		                             & Nested sets of a building set, excluding maximal building blocks; reverse inclusion.                                                       \\
		\addlinespace[1pt]
		\family{generalized}         & Finite-type generalized associahedron~\cite{chapoton2002generalized,hohlweg2011generalized}
		\generalizes{\ref{fam:associahedron}}
		                             & Partial clusters of a finite Coxeter system: compatible sets of almost-positive roots; reverse inclusion.                                  \\
		\addlinespace[1pt]
		\family{pointed}             & Pointed pseudotriangulation polytope~\cite{rote2003pointed}
		\generalizes{\ref{fam:associahedron}}
		                             & Pointed noncrossing graphs on a fixed planar point set, containing all hull edges; reverse inclusion.                                      \\
		\addlinespace[1pt]
		\family{bipermutahedron}     & Bipermutahedron~\cite{ardilaDenhamHuh2023lagrangian,ardila2022bipermutahedron}
		                             & Bisequences of $[n]$ (each label occurs once or twice, some once); merge adjacent parts by union.                                          \\
		\addlinespace[1pt]
		\family{permutree}           & Permutreehedron~\cite{pilaudPons2018permutrees}
		\generalizes{\ref{fam:cube}, \ref{fam:permutahedron}, \ref{fam:associahedron}}
		                             & Schr\"oder permutrees with a fixed decoration; contracting an edge and merging its labels moves upwards.                                   \\
		\addlinespace[1pt]
		\family{zonotope}            & Simple zonotope~\cite{ziegler1995polytopes}
		\generalizes{\ref{fam:cube}, \ref{fam:permutahedron}, \ref{fam:bpermutahedron}}
		                             & Covectors of a simplicial real central arrangement; replacing nonzero signs by zeros moves upwards.                                        \\
		\addlinespace[1pt]
		\family{transportation}      & Generic transportation polytope~\cite{kleeWitzgall1968transportation}
		\generalizes{\ref{fam:simplex}, \ref{fam:cube}}
		                             & Exact bipartite support graphs of feasible tables with fixed generic positive margins; edge inclusion.                                     \\
		\addlinespace[1pt]
		\family{nestedpermutahedron} & Nested permutahedron~\cite{castilloLiu2022nested}
		                             & Chains of ordered partitions of $[d+1]$ with at least two blocks, ordered by reverse inclusion of chains.                                  \\
		\addlinespace[1pt]
		\family{grid}                & Grid associahedron~\cite{palu2021nonkissing}
		\generalizes{\ref{fam:associahedron}}
		                             & Non-kissing collections of bending directed boundary-to-boundary paths in a fixed grid shape; reverse inclusion.                           \\
		\addlinespace[1pt]
		\family{extended}            & Extended nestohedron~\cite{dao2019extended}
		\generalizes{\ref{fam:cube}, \ref{fam:associahedron}}
		                             & Extended nested sets of a building set on $[d]$, ordered by reverse inclusion.                                                             \\
		\addlinespace[1pt]
		\family{signedtree}          & Signed tree associahedron~\cite{pilaud2013signed}
		\generalizes{\ref{fam:associahedron}}
		                             & Signed nested sets of a vertex-signed tree on $d+1$ vertices, ordered by reverse inclusion.                                                \\
		\addlinespace[1pt]
		\family{hypercubegraph}      & Hypercube-graph associahedron~\cite{almeter2022hypercube}
		\generalizes{\ref{fam:cube}, \ref{fam:permutahedron}, \ref{fam:bpermutahedron}, \ref{fam:associahedron}}
		                             & Tubings on a graph on $\pm[d]$ whose union contains no opposite pair, ordered by reverse inclusion.                                        \\
		\addlinespace[1pt]
		\family{allpseudo}           & Pseudotriangulation polytope~\cite{orden2005noncrossing}
		\generalizes{\ref{fam:associahedron}}
		                             & Noncrossing graphs with marks only at pointed vertices, all hull edges and all hull marks; reverse inclusion.                              \\
		\addlinespace[1pt]
		\family{brick}               & Brick polytope (root-independent spherical subword complex)~\cite{pilaudStump2015brick}
		\generalizes{\ref{fam:generalized}}
		                             & Faces of a root-independent spherical subword complex $SC(Q,w_0)$, ordered by reverse inclusion.                                           \\
		\addlinespace[1pt]
		\family{nonkissing}          & Non-kissing associahedron (finite complex)~\cite{palu2021nonkissing}
		\generalizes{\ref{fam:cube}, \ref{fam:grid}}
		                             & Pairwise non-kissing sets of non-self-kissing bending walks (up to reversal) of a gentle quiver; reverse inclusion.                        \\
		\addlinespace[1pt]
		\family{acyclonestohedron}   & Acyclonestohedron~\cite{mantovani2025acyclonestohedra}
		\generalizes{\ref{fam:nestohedron}, \ref{fam:extended}, \ref{fam:hypercubegraph}}
		                             & Acyclic nested sets of a realizable acyclic oriented building set, ordered by reverse inclusion.                                           \\
		\bottomrule
	\end{tabularx}
	\captionsetup{font=small}
	\caption{Polytopes and their face lattices. Cyclic polytopes are simplicial, the rest are simple. Adjoin $\hat 0$ (empty face) in every row; adjoin $\hat 1$ (full polytope) only in row~\ref{fam:cyclic}. Here $d$ is the dimension, and planar point sets are in general position. Selected combinatorial generalizations appear beneath the names.}
	\label{tab:face-models}
\end{table}

\section{Results}\label{sec:results}

Throughout, the \defn{face lattice} of $P$ is $\Lcal(P)$.

Our main result allows us to prescribe the two facets through which the cycle visits the full polytope.

\begin{restatable}[Simplicial polytopes]{theorem}{simplicial}\label{thm:main}
	Let $P$ be a simplicial polytope, and let $F$ and $G$ be distinct facets of $P$. Then the cover graph of $\Lcal(P)$ admits a Hamilton cycle in which the two neighbors of $P$ are $F$ and $G$.
\end{restatable}

Deleting $P$ from this cycle yields a Hamilton path from $F$ to $G$ through all proper faces, including the empty face.

Duality reverses the inclusion order of faces and preserves the undirected cover graph. The dual of a simple polytope is
simplicial; see~\cite{ziegler1995polytopes}. We therefore obtain the following corollary, including the dual endpoint
prescription.

\begin{restatable}[Simple polytopes]{corollary}{simple}\label{cor:simple}
	Let $P$ be a simple polytope, and let $u$ and $v$ be distinct vertices of $P$. Then the cover graph of $\Lcal(P)$ admits a Hamilton cycle in which the two neighbors of the empty face are $u$ and $v$.
\end{restatable}

The proof uses the geometry of a polytope only to obtain a reversible shelling with prescribed first and last facets. The
remaining argument works for any shellable simplicial sphere. When $\Delta$ is a simplicial sphere, let
$\widehat{\Lcal}(\Delta)$ denote its \defn{face poset} with a new maximum $\hat 1$ adjoined.

\begin{restatable}{proposition}{spheres}\label{prop:spheres}
	Let $\Delta$ be a simplicial sphere of dimension at least one, and let $F_1,\ldots,F_N$ be a shelling of $\Delta$. Then the cover graph of $\widehat{\Lcal}(\Delta)$ admits a Hamilton cycle in which the two neighbors of $\hat 1$ are $F_1$ and $F_N$.
\end{restatable}

Note that any shelling of a simplicial sphere is reversible; see~\cite[Section~3]{novik2023faces} and, for the polytopal case,
\cite[Lemma~8.10]{ziegler1995polytopes}. For a survey of shellability, see Danaraj and Klee~\cite{danaraj1978shellable}.

If the polytope is not simplicial, then the decomposition of its face lattice still works, but the parts are not cubes anymore.
Using a simple induction, this still gives a perfect matching in the face lattice of any polytope.

\begin{remark}
	Our use of shellings to construct matchings is closely
	related to Chari's construction of discrete Morse
	matchings for shellable pseudomanifolds~\cite[Proposition~4.1]{chari2000discrete}.
	In particular, the following theorem is a consequence
	of their result. The proof below presents essentially
	the same shelling induction in the language of perfect
	matchings, to make this connection explicit. Chari's
	construction additionally gives acyclicity in the
	sense of discrete Morse theory.
\end{remark}

\begin{restatable}[{Chari~\cite[Proposition~4.1]{chari2000discrete}}]{theorem}{perfmatchings}\label{thm:perfmatchings}
	Let $P$ be a polytope, and let $F_1,\ldots,F_N$ be a shelling of its boundary. For $1\le k<N$, let $K_k=\bigcup_{i=1}^k\Lcal(F_i)$ be the set of faces contained in the first $k$ facets, including the empty face. Then the subgraphs of the cover graph of $\Lcal(P)$ induced by $K_k$ and by $\Lcal(P)\setminus K_k$ each admit a perfect matching. The cover graph of $\Lcal(P)$ itself also admits a perfect matching.
\end{restatable}

Using this, we are able to slightly extend \Cref{thm:main}:

\begin{restatable}{corollary}{exceptionalface}\label{cor:exceptionalface}
	Let $P$ be a polytope and let $F$ be a facet of $P$. Assume that $P$ is simplicial with the possible exception of $F$. Then $\Lcal(P)$ admits a Hamilton cycle.
\end{restatable}

\section{Preliminaries}\label{sec:prelim}

Throughout this paper, we consider only positive-dimensional polytopes.

\subsection{Posets, Polytopes, and Cubes}

Given a finite poset $(\mathcal P,\le)$, a pair $x<y$ is a \defn{cover relation} if there is no $z\in\mathcal P$ with $x<z<y$. The
\defn{cover graph} has the elements of $\mathcal P$ as vertices and the cover relations as undirected edges. For $x\le y$, the
\defn{interval} from $x$ to $y$ is
\[
	[x,y]=\{z\in\mathcal P:x\le z\le y\}.
\]
The \defn{Boolean lattice} on a set $S$ consists of all subsets of $S$, ordered by inclusion. Its cover graph is the
  $|S|$-dimensional cube: an edge adds or removes one element. More generally, the interval $[R,S]$ is a Boolean lattice on the
  free elements $S\setminus R$, and its cover graph is a cube of dimension $|S|-|R|$.

The \defn{face lattice} $\Lcal(P)$ of a polytope $P$ orders all its faces by inclusion. We set $\dim\emptyset=-1$, so its cover
relations are precisely inclusions between faces whose dimensions differ by one. If $P$ has dimension $d$, its faces of dimension
$d-1$ are called \defn{facets}, and its faces of dimension $d-2$ are called \defn{ridges}. Every ridge belongs to exactly two
facets~\cite{ziegler1995polytopes}.

A polytope of dimension $d$ is \defn{simplicial} if each facet is a simplex, and \defn{simple} if each vertex is incident to
exactly $d$ edges. We identify each proper face of a simplicial polytope with its vertex set. In particular, for a facet $F$, the
family $2^F$ consists of all faces contained in $F$. For $R\subseteq F$, the cover graph of the interval $[R,F]$ is a cube. A
drawing in which a face of rank $k=\dim A+1$ is placed at height $k$ is a \defn{diagram} of the face lattice.

\subsection{Shellings and Restriction Faces}

A \defn{polytopal complex} is a finite collection of polytopes, including the empty face, that contains every face of each of its
members and in which the intersection of any two members is a face of both. Its maximal members are its \defn{facets}. It is
\defn{pure} if all facets have the same dimension. The proper faces of a polytope $P$ form its \defn{boundary complex}, denoted by
$\partial P$.

Shellings of pure polytopal complexes are defined recursively; see~\cite[Definition~8.1]{ziegler1995polytopes}. Let
$F_1,\ldots,F_N$ be the facets of such a complex. If the facets are points, every ordering is a \defn{shelling}. Otherwise, the
ordering is a shelling if $\partial F_1$ has a shelling and, for every $i>1$, there is a shelling $G_1,\ldots,G_t$ of $\partial
F_i$ and an index $s$ with $1\le s\le t$ such that
\begin{equation}\label{eq:polytopal-shelling}
	\Lcal(F_i)\cap\bigcup_{j<i}\Lcal(F_j)
	=\bigcup_{\ell=1}^{s}\Lcal(G_\ell).
\end{equation}
Thus, the old faces of $F_i$ form an initial part of a
shelling of its boundary. For simplicial complexes, this
reduces to the following definition.

A \defn{simplicial complex} is a family of finite sets, called faces, closed under taking subsets. Its maximal faces are its
facets. It is \defn{pure} if all facets have the same size. A face consisting of $d+1$ elements has dimension $d$. The proper
faces of a simplicial polytope of dimension $d$ form its \defn{boundary complex}, a pure simplicial complex whose geometric
realization is a sphere of dimension $d-1$.

Let $F_1,\ldots,F_N$ be the facets of a pure simplicial complex, each of size $d\ge 2$. The ordering is a \defn{shelling} if, for
every $i>1$, the faces of $F_i$ already present form a subcomplex generated by a nonempty collection of facets of $F_i$, which are
ridges of $P$ in the polytopal case. Equivalently, there is a nonempty set $R_i\subseteq F_i$ such that
\begin{equation}\label{eq:shelling}
	2^{F_i}\cap\bigcup_{j<i}2^{F_j}
	=\bigcup_{v\in R_i}2^{F_i\setminus\{v\}}.
\end{equation}
We set $R_1=\emptyset$. The set $R_i$ is called the \defn{restriction face} of $F_i$. Explicitly,
\begin{equation}\label{eq:restriction}
	R_i=\{v\in F_i:F_i\setminus\{v\}\subseteq F_j
	\text{ for some }j<i\}.
\end{equation}
Thus, a vertex belongs to $R_i$ when the opposite facet of $F_i$ has appeared earlier. In the tetrahedron example, the restriction face of $F_3=134$ is the edge $34$, because the two old ridges are $13$ and $14$. See also \Cref{fig:interval}.

\begin{figure}[htbp]
	\centering
	\input{figures/restriction}
	\caption{The restriction face for $F_i=134$ in the tetrahedron shelling. Left: the old ridges $13$ and $14$ are orange, and their opposite vertices form $R_i=34$ (purple). Right: the Boolean lattice $2^{\{1,3,4\}}$. The old faces are orange and the interval $[34,134]$ of new faces is blue. The only free element in this interval is $1$, so its cover graph is an edge.}
	\label{fig:interval}
\end{figure}

A shelling is \defn{reversible} if its reverse is also a shelling. Every shelling of the boundary of a polytope is reversible,
without a simpliciality assumption~\cite[Lemma~8.10]{ziegler1995polytopes}. Bruggesser and Mani~\cite{bruggesser1971shellable}
construct reversible shellings of polytopes using a line in general position. Starting inside the polytope, traverse the line in
one direction and order the facets by the intersections with their supporting hyperplanes. After passing through infinity,
continue from the opposite direction back to the starting point. This gives a \defn{line shelling}. Reversing the direction
reverses the ordering. Choosing a generic line through the relative interiors of two prescribed facets makes them the first and
last facets; see also~\cite[Corollary~8.13]{ziegler1995polytopes}.

\Needspace{7\baselineskip}
We will use the following form of their theorem.

\begin{theorem}[Line shellings, Bruggesser and Mani~\cite{bruggesser1971shellable}]\label{thm:line}
	Let $P$ be a polytope of dimension at least two, and let $F$ and $G$ be distinct facets. There exists a reversible shelling $F_1,\ldots,F_N$ of its boundary with $F_1=F$ and $F_N=G$.
\end{theorem}

When $P$ is simplicial, the first facet has no old ridges and the last facet has all its facets present. Every intermediate facet
has both an earlier and a later neighboring facet. In terms of restriction faces, this says
\begin{equation}\label{eq:nontrivial}
	R_1=\emptyset,\qquad
	\emptyset\ne R_i\subsetneq F_i\quad(1<i<N),\qquad
	R_N=F_N.
\end{equation}
Indeed, the forward shelling gives an old ridge for every $i>1$, and the reverse shelling gives a ridge shared with a later facet for every $i<N$. Since a ridge belongs to exactly two facets, this latter ridge is not old in the forward shelling. The same argument applies to a simplicial sphere with a shelling.

\section{Related Work}\label{sec:related}

\paragraph{Listing faces of polytopes.}
In~\cite{Listingfaces}, Behrooznia et al. established Hamiltonicity of face lattices for hypercubes, permutahedra and
$B$-permutahedra, associahedra, cyclic polytopes, three-dimensional polytopes, graph associahedra of chordal graphs, and
quotientopes. Our result subsumes the simplicial and simple cases. In particular, it applies to cyclic polytopes, which are
simplicial~\cite{ziegler1995polytopes}, and to all graph associahedra, which are simple~\cite{Graph_associahedra_general}. Their
three-dimensional result also covers polytopes that are neither simplicial nor simple, so it is not subsumed by our theorem. Their
quotientope result likewise extends beyond the simple case.

Their proof for three-dimensional polytopes in~\cite{Listingfaces} builds a listing by adding faces in a suitable order. In
Section~12 of the \href{https://tmuetze.de/papers/flattice.pdf}{full version}, the authors explicitly suggest shellings as a
possible approach to higher dimensions.

\paragraph{Shellings and discrete Morse theory.}
Chari~\cite{chari2000discrete} studies shellings in the setting of Forman's discrete Morse theory for regular cell complexes.
In this theory, a discrete Morse function describes the homotopy type of a complex by a cell complex with one cell for each
critical cell. Chari describes such functions using matchings on the Hasse diagram: critical cells are precisely the unmatched
elements, and reversing the matched edges in the downward orientation gives an acyclic directed graph. For a shelling
$F_1,\ldots,F_N$ of $\partial P$, his construction leaves only a vertex $v$ of $F_1$ and the last facet $F_N$ unmatched, while
its restriction to every proper shelling prefix leaves only $v$ unmatched~\cite[Proposition~4.1]{chari2000discrete}. Adding the
pairs $\{\emptyset,v\}$ and $\{F_N,P\}$ therefore gives a perfect matching of $\Lcal(P)$. Since no matching edge crosses a
shelling prefix, this matching restricts to perfect matchings of each $K_k$ and its complement. Thus,
\Cref{thm:perfmatchings} is an immediate consequence of Chari's result. Our proof follows his shelling induction, keeping
only the matching construction and omitting the acyclicity argument.

\paragraph{Perfect matchings in cubes.}
Fink~\cite{fink2007perfect} proved that every perfect matching of a cube of dimension at least two extends to a Hamilton cycle.
Gregor~\cite{gregor2009perfect} strengthened this to several cubes, with the additional requirement that the prescribed matching
connects them.

\begin{theorem}[Gregor~\cite{gregor2009perfect}]\label{thm:gregor}
	Let $Q_1,\ldots,Q_m$ be vertex-disjoint cubes of positive dimensions, with at least four vertices in total. Write $V=\bigcup_i V(Q_i)$ and $E=\bigcup_i E(Q_i)$. Let $M$ be a perfect matching of the complete graph on $V$. If $(V,E\cup M)$ is connected, then there is a perfect matching $M'\subseteq E$ such that $M\cup M'$ is a Hamilton cycle.
\end{theorem}

The edges of $M$ may lie inside cubes or join different cubes. In our application, all of them will be cover relations of the face
lattice.

\section{Proofs}\label{sec:proofs}

\subsection{Simplicial Polytopes}

We now prove \Cref{thm:main}. \simplicial*

For a segment with vertices $u$ and $v$, the cycle $\emptyset,u,P,v,\emptyset$ proves the assertion. Let $P$ have dimension $d\ge
2$. By \Cref{thm:line}, choose a reversible shelling $F_1,\ldots,F_N$ of its boundary with $F_1=F$ and $F_N=G$, where $F,G$ are
the prescribed facets.

\subsubsection{The Cube Partitions}

The defining property of a restriction face is that it is the unique smallest new face of its facet;
see~\cite[Chapter~8]{ziegler1995polytopes}.

\begin{lemma}[New faces]\label{lem:new}
	For every $i\in\{1,\ldots,N\}$, the faces appearing for the first time in $F_i$ are precisely the interval $[R_i,F_i]$. That is,
	\[
		2^{F_i}\setminus\bigcup_{j<i}2^{F_j}=[R_i,F_i].
	\]
\end{lemma}
\begin{proof}
	For $i=1$, all subsets of $F_1$ are new and $R_1=\emptyset$. For $i>1$, a face $A\subseteq F_i$ is old if and only if it is contained in an old ridge, by the shelling condition~\eqref{eq:shelling}. The old ridges are exactly $F_i\setminus\{v\}$ with $v\in R_i$. Thus $A$ is old if and only if it omits some vertex of $R_i$, and it is new if and only if $R_i\subseteq A$.
\end{proof}

Assigning each proper face to the first facet in which it appears partitions the boundary faces into these intervals. The last
interval is the singleton $[F_N,F_N]$. We pair it with $P$ and define
\begin{equation}\label{eq:forward}
	B_i=[R_i,F_i]\quad(1\le i<N),\qquad B_N=\{F_N,P\}.
\end{equation}
By~\eqref{eq:nontrivial}, the graphs $Q_i$ induced by $B_i$ in the cover graph of $\Lcal(P)$ are cubes of dimension $d-|R_i|>0$ for $i<N$, and $Q_N$ is an edge. Consequently, the $B_i$ partition all faces into positive-dimensional cubes. In particular, the empty face belongs to $B_1$ and the full polytope belongs to $B_N$.

Next, read the shelling backwards. We keep the original facet indices: $R'_i$ denotes the restriction face of $F_i$ in the reverse
ordering.

\begin{lemma}[Complementary restriction faces]\label{lem:complement}
	For every $i\in\{1,\ldots,N\}$, we have $R'_i=F_i\setminus R_i$.
\end{lemma}
\begin{proof}
	Fix a vertex $v\in F_i$. The ridge $F_i\setminus\{v\}$ is contained in exactly one other facet, say $F_j$. If $j<i$, then $v\in R_i$ and $v\notin R'_i$. If $j>i$, the reverse holds. These are the only possibilities, so each vertex of $F_i$ belongs to exactly one of $R_i$ and $R'_i$. Note that the relevant property, that a ridge is contained in exactly two facets, is true for simplicial spheres. However, it fails for more general simplicial complexes, even if they have a reversible shelling.
\end{proof}

Applying \Cref{lem:new} to the reverse shelling and again pairing its last facet with $P$, we obtain a second partition
\begin{equation}\label{eq:reverse}
	B'_1=\{F_1,P\},\qquad
	B'_i=[F_i\setminus R_i,F_i]\quad(2\le i\le N).
\end{equation}
Its cubes have dimensions $1,|R_2|,\ldots,|R_N|$, respectively, which are all positive. Note that $P$ has changed its partner from $F_N$ to $F_1$, and the empty face has moved from the first facet to the last.

\subsubsection{A Matching from the Reverse Shelling}

For every $i\in\{2,\ldots,N\}$, choose a vertex $r_i\in R_i$. In the reverse interval $B'_i=[F_i\setminus R_i,F_i]$, this is a
free element, so adding or removing it stays within the interval. Define
\begin{equation}\label{eq:matching}
	M_i=\bigl\{\{A,A\cup\{r_i\}\}: A\in B'_i,\ r_i\notin A\bigr\},
	\qquad
	M=\{\{F_1,P\}\}\cup\bigcup_{i=2}^N M_i.
\end{equation}
Each $M_i$ is the parallel matching in the direction $r_i$. Every face in $B'_i$ belongs to exactly one of its edges, and each edge is a cover relation. Since the reverse cubes partition the faces, $M$ is a perfect matching of the cover graph of $\Lcal(P)$.

For $i>1$, the matching edge incident to the facet $F_i$ is
\begin{equation}\label{eq:descending}
	\{F_i,F_i\setminus\{r_i\}\}\in M.
\end{equation}
Its lower endpoint is an old ridge in the forward shelling, so its first containing facet has index $j<i$. Thus, this matching edge leaves $B_i$ and enters $B_j$. Following these edges strictly decreases the index until reaching $B_1$, so the cube edges together with $M$ form a connected graph. The cubes have positive dimensions and at least four vertices in total. By \Cref{thm:gregor}, there is a perfect matching $M'\subseteq\bigcup_i E(Q_i)$ such that $M\cup M'$ is a Hamilton cycle of $\Lcal(P)$. The edge $\{F_1,P\}$ lies in $M$, while $\{F_N,P\}$ is the only edge of $Q_N$ and is therefore forced into $M'$. Hence the neighbors of $P$ are the prescribed facets $F$ and $G$, completing the proof of \Cref{thm:main}.

\begin{figure}[htbp]
	\centering
	\input{figures/connections}
	\caption{The matching edges incident to $F_2,F_3,F_4$ connect the $B_i$ in the tetrahedron example. The choices $r_2=4$, $r_3=3$, $r_4=2$ give $124$--$12$, $134$--$14$, and $234$--$34$. Each edge joins a facet in one cube to an old ridge in an earlier cube. After contraction, these three edges form the path $B_4$--$B_3$--$B_2$--$B_1$. The edge $P$--$123$ is also in $M$. Internal cube edges and the other vertices are omitted.}
	\label{fig:connections}
\end{figure}

\begin{figure}[htbp]
	\centering
	\input{figures/cycle}
	\caption{A Hamilton cycle obtained from the two matchings in the tetrahedron example. The matching $M$ is orange and the completion $M'$ is blue. Heights are ranks, and the two occurrences of $123$ represent the same face, closing the cycle. Every blue edge lies in one forward cube. At the full tetrahedron $P$, the orange edge goes to $F_1=123$ and the blue edge goes to $F_4=234$, as required.}
	\label{fig:cycle}
\end{figure}

\simple*
\begin{proof}
	Apply \Cref{thm:main} to the simplicial dual of $P$, prescribing the facets corresponding to $u$ and $v$. Duality gives an isomorphism of cover graphs that sends the full dual polytope to the empty face of $P$ and the prescribed facets to $u$ and $v$.
\end{proof}

\subsection{Simplicial Spheres}

A \defn{simplicial sphere} is a finite simplicial complex whose geometric realization is homeomorphic to a sphere. Its face poset
contains the empty face but has no greatest face. Let $\widehat{\Lcal}(\Delta)$ denote this poset with a new maximum $\hat 1$
adjoined. The new cover relations join $\hat 1$ to every facet.

\spheres*
\begin{proof}
	Every ridge of a simplicial sphere belongs to exactly two facets. Together with the forward and reverse shelling conditions, this gives~\eqref{eq:nontrivial} and the complementary restriction faces of \Cref{lem:complement}. The new-face argument in \Cref{lem:new} uses only the shelling condition. Consequently, replacing $P$ by $\hat 1$ in~\eqref{eq:forward} and~\eqref{eq:reverse} again gives two partitions into positive-dimensional cubes. The parallel matching~\eqref{eq:matching} is perfect, and the decreasing-index argument following~\eqref{eq:descending} proves connectivity. Applying \Cref{thm:gregor} completes the matching, with the edges $\{F_1,\hat 1\}$ and $\{F_N,\hat 1\}$ forced as before.
\end{proof}

The zero-dimensional sphere consists of two vertices and gives the same four-cycle as a segment. For general simplicial spheres,
the proposition prescribes the ends of the given shelling. The freedom to choose any two distinct facets in \Cref{thm:main} comes
from the line-shelling theorem for polytopes.

\subsection{One Exceptional Face}

Let us now investigate how \Cref{thm:perfmatchings} allows us to have an exceptional, non-simplex facet.

\exceptionalface*
\begin{proof}

	The case of a segment follows from \Cref{thm:main}, so assume that $\dim P\ge2$. By \Cref{thm:line}, choose a line shelling
	$F_1,\ldots,F_N$ with $F_N=F$. Every $F_i$ with $i<N$ is a simplex, so its new faces form the Boolean interval
	$B_i=[R_i,F_i]$, with $R_i$ defined by~\eqref{eq:restriction}. These intervals have positive dimensions, since each $F_i$ with
	$i<N$ has a ridge shared with a later facet. Every proper face of $F_N$ has already appeared, so we again set $B_N=\{F_N,P\}$.
	Thus, the forward blocks induce positive-dimensional cubes $Q_1,\ldots,Q_N$.

	In the reverse shelling, the first block is $\Lcal(F_N)$. For $2\le i<N$, the same argument as in \Cref{lem:complement} gives
	the reverse block $B'_i=[F_i\setminus R_i,F_i]$. Pairing the final facet $F_1$ with $P$, we obtain the partition
	\[
		B'_1=\{F_1,P\},\qquad
		B'_i=[F_i\setminus R_i,F_i]\quad(2\le i<N),\qquad
		B'_N=\Lcal(F_N).
	\]
	By \Cref{thm:perfmatchings}, choose a perfect matching $M_N$ of the cover graph of $\Lcal(F_N)$. For $2\le i<N$, choose
	$r_i\in R_i$ and let $M_i$ be the parallel matching in direction $r_i$ on $B'_i$, as in~\eqref{eq:matching}. Together, these
	matchings give a perfect matching of the cover graph of $\Lcal(P)$:
	\[
		M=\{\{F_1,P\}\}\cup\bigcup_{i=2}^{N-1}M_i\cup M_N.
	\]

	For $2\le i<N$, the matching edge incident to $F_i$ joins it to the old ridge $F_i\setminus\{r_i\}$, and hence connects $B_i$
	to an earlier block. The matching $M_N$ pairs $F_N$ with a facet $H$ of $F_N$, which is a ridge of $P$. The other facet of $P$
	containing $H$ is earlier than $F_N$, so this edge also connects $B_N$ to an earlier block. Following these edges strictly
	decreases the index until reaching $B_1$. Thus, the cube edges together with $M$ form a connected graph, and \Cref{thm:gregor}
	completes $M$ to a Hamilton cycle. In this cycle, the neighbors of $P$ are $F_1$ and $F_N=F$.
\end{proof}

\subsection{Perfect Matchings}

For completeness, we present Chari's shelling induction~\cite[Proposition~4.1]{chari2000discrete} in the language of perfect
matchings to prove \Cref{thm:perfmatchings}. \perfmatchings*
\begin{proof}

	We prove all three assertions simultaneously by induction on $d=\dim P$. If $d=1$, let $F_1=u$ and $F_2=v$ be the vertices of
	the segment. The pairs $\{\emptyset,u\}$ and $\{v,P\}$ match $K_1$ and its complement, respectively, and together match
	$\Lcal(P)$.

	Now let $d\ge2$ and assume the assertions hold in smaller positive dimensions. Set $K_0=\varnothing$, meaning the empty
	collection of faces, and define
	\[
		D_i=\Lcal(F_i)\setminus K_{i-1}
		\quad(1\le i\le N).
	\]
	We first match each $D_i$ with $i<N$. For $i=1$, we have $D_1=\Lcal(F_1)$, which has a perfect matching by induction.

	For $1<i<N$, the shelling condition~\eqref{eq:polytopal-shelling} gives a shelling $G_1,\ldots,G_t$ of $\partial F_i$ such
	that the old faces of $F_i$ are
	\[
		A_i=\Lcal(F_i)\cap K_{i-1}
		=\bigcup_{\ell=1}^{s}\Lcal(G_\ell).
	\]
	Here $s<t$. Indeed, the reverse shelling gives a ridge of $P$ contained in $F_i$ and a later facet. Since a ridge belongs to
	exactly two facets, this ridge is not old in the forward shelling. Thus, $A_i$ is a proper initial part of a shelling of
	$\partial F_i$. Applying the assertion about complements in dimension $d-1$ gives a perfect matching of
	\[
		D_i=\Lcal(F_i)\setminus A_i.
	\]
	Every edge of this matching is a cover relation in $\Lcal(P)$, since the dimensions of the two faces differ by one.

	Every proper face of $F_N$ is contained in a ridge of $P$, whose other containing facet appears earlier. Hence $D_N=\{F_N\}$,
	and we pair $F_N$ with $P$. For $1\le k<N$, the disjoint decompositions
	\[
		K_k=\bigsqcup_{i=1}^{k}D_i,
		\qquad
		\Lcal(P)\setminus K_k
		=\left(\bigsqcup_{i=k+1}^{N-1}D_i\right)
		\sqcup\{F_N,P\}
	\]
	show that the matchings constructed above give perfect matchings of $K_k$ and its complement. Their union is a perfect
	matching of $\Lcal(P)$, completing the induction.
\end{proof}

The restriction $k<N$ is necessary. The full boundary, including the empty face, has an odd number of faces by Euler's relation
and therefore has no perfect matching. The last facet is instead matched to $P$.

\section{Conclusion and Open Problems}\label{sec:conclusion}

We have shown that the face lattice of every simplicial or simple polytope is Hamiltonian. For simplicial polytopes, the two
neighbors of the full polytope can be prescribed arbitrarily. The proof combines two decompositions coming from opposite
shellings: the forward decomposition supplies cubes, and the reverse decomposition supplies a perfect matching connecting them.

The argument also allows a single facet that is
	not a simplex. Choosing it as the final facet of a line
	shelling preserves the forward partition into cubes. In
	the reverse partition, Chari's construction supplies a
	perfect matching of the exceptional facet's face lattice,
	and the same connectivity argument applies.

The natural next goal is the conjecture of Behrooznia et al.~\cite{Listingfaces} for all polytopes. Shellings exist in this
generality~\cite{bruggesser1971shellable}, but a nonsimplicial facet does not have a Boolean face lattice. Its new faces therefore
need not induce a cube. The shelling induction in \Cref{thm:perfmatchings} supplies perfect matchings on these pieces, but
the completion theorem we use applies only to cubes. One possible approach is to find other classes of pieces for which a
matching-extension theorem like \Cref{thm:gregor} holds.

The extension to simplicial spheres also separates the combinatorial part of the proof from its geometric input. It would be
interesting to obtain comparable Hamiltonicity results for simplicial spheres without assuming shellability. A further question is
whether the existence proof can be turned into a direct face-generation algorithm with a useful bound on the time between
consecutive faces. Fink~\cite{fink2020algorithms} gives an algorithm, linear in the number of vertices, for extending a perfect
matching of a cube to a Hamilton cycle, as well as an incremental algorithm. The present argument proves existence and the
endpoint prescription, but does not provide such a bound.

\Cref{fig:octahedron} collects all steps of the construction for an octahedron.

\textbf{Acknowledgements.} We thank Torsten Mütze for helpful discussions on this topic. Furthermore, we disclose that an LLM (gpt-6 astra) was used to proofread and create the TikZ code for
all figures from a clear description. The LLM was also used extensively to find examples and their citations for \Cref{tab:face-models}. The citations were verified.

\clearpage
\begin{figure}[p]
	\centering
	\input{figures/octahedron}
	\caption{An octahedron with opposite pairs $1,4$, $2,5$, and $3,6$. (a) The shelling and restriction data. (b) All 28 faces in the forward cubes; solid orange curves form $M$, dashed blue edges form $M'$, and thin gray edges are the remaining cube edges. Other inter-cube covers are omitted. White gaps distinguish crossings, not incidences. (c) The reverse cubes supporting $M$. (d) Selected edges of $M$ connect the forward cubes. The union $M\cup M'$ is a Hamilton cycle, with $F_1$ and $F_8$ adjacent to $P$.}
	\label{fig:octahedron}
\end{figure}

\clearpage

\begin{refcontext}[sorting=none]
	\printbibliography
\end{refcontext}
\end{document}

%% file: figures/shelling.tex
\begin{tikzpicture}[line cap=round,line join=round]
\foreach \i in {1,2,3,4}{
 \begin{scope}[xshift={(\i-1)*3.65cm}]
  \coordinate (a) at (0,0);
  \coordinate (b) at (2.7,0);
  \coordinate (c) at (1.35,2.25);
  \coordinate (d) at (1.35,.82);
  \ifnum\i=1
   \fill[Blue!15,even odd rule] (-.23,-.16)--(2.93,-.16)--(1.35,2.6)--cycle
      (a)--(b)--(c)--cycle;
  \fi
  \ifnum\i=2 \fill[Blue!15] (a)--(b)--(d)--cycle;\fi
  \ifnum\i=3 \fill[Blue!15] (a)--(c)--(d)--cycle;\fi
  \ifnum\i=4 \fill[Blue!15] (b)--(c)--(d)--cycle;\fi
  \draw[black!40] (a)--(b)--(c)--cycle (a)--(d)--(b) (c)--(d);
  \ifnum\i=1 \draw[Blue,thick] (a)--(b)--(c)--cycle;\fi
  \ifnum\i=2
   \draw[Blue,thick] (a)--(d)--(b);
   \draw[matched] (a)--(b);
  \fi
  \ifnum\i=3
   \draw[matched] (c)--(a)--(d);
   \draw[Purple,line width=2pt] (c)--(d);
  \fi
  \ifnum\i=4
   \fill[Purple!22] (b)--(c)--(d)--cycle;
   \draw[matched] (b)--(c)--(d)--cycle;
  \fi
  \foreach \p/\lab/\pos in {a/1/below left,b/2/below right,c/3/above,d/4/below}{
   \node[facepoint] at (\p) {};
   \node[\pos,font=\small,inner sep=3pt] at (\p) {$\lab$};
  }
  \ifnum\i=2 \draw[Purple,line width=1.2pt] (d) circle (3.4pt);\fi
  \node[font=\small] at (1.35,-.65) {$F_\i$};
  \ifnum\i=1 \node[font=\small,Purple!85!black] at (1.35,-1.1) {$R_1=\emptyset$};\fi
  \ifnum\i=2 \node[font=\small,Purple!85!black] at (1.35,-1.1) {$R_2=4$};\fi
  \ifnum\i=3 \node[font=\small,Purple!85!black] at (1.35,-1.1) {$R_3=34$};\fi
  \ifnum\i=4 \node[font=\small,Purple!85!black] at (1.35,-1.1) {$R_4=234$};\fi
 \end{scope}
}
\end{tikzpicture}

%% file: figures/restriction.tex
\begin{tikzpicture}[line cap=round,line join=round]
\begin{scope}[xshift=.3cm,yshift=.3cm]
 \coordinate (a) at (0,0);
 \coordinate (b) at (3.2,0);
 \coordinate (c) at (1.6,2.7);
 \fill[Blue!9] (a)--(b)--(c)--cycle;
 \draw[matched] (a)--(c)--(b);
 \draw[Purple,line width=1.6pt] (a)--(b);
 \node[Purple,below] at (1.6,0) {$R_i=34$};
 \node at (1.6,1.1) {$F_i=134$};
 \foreach \p/\lab/\pos in {a/3/below left,b/4/below right,c/1/above}{
  \node[facepoint] at (\p) {};
  \node[\pos] at (\p) {$\lab$};
 }
\end{scope}
\draw[-{Latex[length=2mm]},black!65] (4.45,1.55)--(5.25,1.55);
\begin{scope}[xshift=6.3cm,x=1.65cm,y=1.05cm]
 \coordinate (e) at (1,0);
 \coordinate (a) at (0,1); \coordinate (b) at (1,1); \coordinate (c) at (2,1);
 \coordinate (ab) at (0,2); \coordinate (ac) at (1,2); \coordinate (bc) at (2,2);
 \coordinate (abc) at (1,3);
 \foreach \y in {0,1,2,3} \draw[rankline] (-.45,\y)--(2.75,\y);
 \foreach \u/\v in {e/a,e/b,e/c,a/ab,a/ac,b/ab,b/bc,c/ac,c/bc,ab/abc,ac/abc,bc/abc}
  \draw[cover] (\u)--(\v);
 \foreach \u/\v in {e/a,e/b,e/c,a/ab,a/ac,b/ab,c/ac} \draw[Orange,line width=.9pt] (\u)--(\v);
 \draw[completed] (bc)--(abc);
 \foreach \p/\lab in {e/\emptyset,a/1,b/3,c/4,ab/13,ac/14}
  \node[facelabel,text=Orange!85!black] at (\p) {$\lab$};
 \node[facelabel,text=Blue] at (bc) {$34$};
 \node[facelabel,text=Blue] at (abc) {$134$};
 \node[Blue,right,font=\small] at (2.65,2.55) {$[34,134]$};
\end{scope}
\end{tikzpicture}

%% file: figures/connections.tex
\begin{tikzpicture}[line cap=round,line join=round]
\foreach \i/\lo/\hi in {1/12/123,2/14/124,3/34/134,4/234/P}{
 \begin{scope}[xshift={(\i-1)*3.45cm}]
  \draw[cubebox] (-.43,-.43) rectangle (1.57,1.55);
  \coordinate (lo\i) at (0,0);
  \coordinate (hi\i) at (1.14,1.1);
  \node[font=\small] at (.57,-.85) {$B_\i$};
 \end{scope}
}
\draw[matched] (hi2)--(lo1);
\draw[matched] (hi3)--(lo2);
\draw[matched] (lo4) to[out=205,in=-25] (lo3);
\draw[matched] (hi4) .. controls (10,2.7) and (2.6,2.7) ..
 node[above,text=black,font=\small,pos=.5] {$P$--$F_1$} (hi1);
\foreach \i in {1,2,3,4} {
 \node[facepoint] at (lo\i) {};
 \node[facepoint] at (hi\i) {};
}
\foreach \i/\lo/\hi in {1/12/123,2/14/124,3/34/134,4/234/P}{
 \node[below,font=\small,inner sep=4pt] at (lo\i) {$\lo$};
 \node[below right,font=\small,inner sep=3pt] at (hi\i) {$\hi$};
}
\end{tikzpicture}

%% file: figures/cycle.tex
\begin{tikzpicture}[x=.83cm,y=.72cm,line cap=round,line join=round]
\draw[rankline] (-.3,0)--(16.4,0);
\node[font=\scriptsize,left,text=black!60] at (-.4,0) {$ 0 $};
\draw[rankline] (-.3,1)--(16.4,1);
\node[font=\scriptsize,left,text=black!60] at (-.4,1) {$ 1 $};
\draw[rankline] (-.3,2)--(16.4,2);
\node[font=\scriptsize,left,text=black!60] at (-.4,2) {$ 2 $};
\draw[rankline] (-.3,3)--(16.4,3);
\node[font=\scriptsize,left,text=black!60] at (-.4,3) {$ 3 $};
\draw[rankline] (-.3,4)--(16.4,4);
\node[font=\scriptsize,left,text=black!60] at (-.4,4) {$ 4 $};
\coordinate (n0) at (0,3);
\coordinate (n1) at (1,4);
\coordinate (n2) at (2,3);
\coordinate (n3) at (3,2);
\coordinate (n4) at (4,3);
\coordinate (n5) at (5,2);
\coordinate (n6) at (6,1);
\coordinate (n7) at (7,2);
\coordinate (n8) at (8,3);
\coordinate (n9) at (9,2);
\coordinate (n10) at (10,1);
\coordinate (n11) at (11,0);
\coordinate (n12) at (12,1);
\coordinate (n13) at (13,2);
\coordinate (n14) at (14,1);
\coordinate (n15) at (15,2);
\coordinate (n16) at (16,3);
\draw[matched] (n0)--(n1);
\draw[completed] (n1)--(n2);
\draw[matched] (n2)--(n3);
\draw[completed] (n3)--(n4);
\draw[matched] (n4)--(n5);
\draw[completed] (n5)--(n6);
\draw[matched] (n6)--(n7);
\draw[completed] (n7)--(n8);
\draw[matched] (n8)--(n9);
\draw[completed] (n9)--(n10);
\draw[matched] (n10)--(n11);
\draw[completed] (n11)--(n12);
\draw[matched] (n12)--(n13);
\draw[completed] (n13)--(n14);
\draw[matched] (n14)--(n15);
\draw[completed] (n15)--(n16);
\node[facepoint] at (n0) {};
\node[above,font=\small,inner sep=4pt] at (n0) {$ 123 $};
\node[facepoint] at (n1) {};
\node[above,font=\small,inner sep=4pt] at (n1) {$ P $};
\node[facepoint] at (n2) {};
\node[above,font=\small,inner sep=4pt] at (n2) {$ 234 $};
\node[facepoint] at (n3) {};
\node[above,font=\small,inner sep=4pt] at (n3) {$ 34 $};
\node[facepoint] at (n4) {};
\node[above,font=\small,inner sep=4pt] at (n4) {$ 134 $};
\node[facepoint] at (n5) {};
\node[above,font=\small,inner sep=4pt] at (n5) {$ 14 $};
\node[facepoint] at (n6) {};
\node[below,font=\small,inner sep=4pt] at (n6) {$ 4 $};
\node[facepoint] at (n7) {};
\node[above,font=\small,inner sep=4pt] at (n7) {$ 24 $};
\node[facepoint] at (n8) {};
\node[above,font=\small,inner sep=4pt] at (n8) {$ 124 $};
\node[facepoint] at (n9) {};
\node[above,font=\small,inner sep=4pt] at (n9) {$ 12 $};
\node[facepoint] at (n10) {};
\node[below,font=\small,inner sep=4pt] at (n10) {$ 2 $};
\node[facepoint] at (n11) {};
\node[below,font=\small,inner sep=4pt] at (n11) {$ \emptyset $};
\node[facepoint] at (n12) {};
\node[below,font=\small,inner sep=4pt] at (n12) {$ 1 $};
\node[facepoint] at (n13) {};
\node[above,font=\small,inner sep=4pt] at (n13) {$ 13 $};
\node[facepoint] at (n14) {};
\node[below,font=\small,inner sep=4pt] at (n14) {$ 3 $};
\node[facepoint] at (n15) {};
\node[above,font=\small,inner sep=4pt] at (n15) {$ 23 $};
\node[facepoint] at (n16) {};
\node[above,font=\small,inner sep=4pt] at (n16) {$ 123 $};
\draw[matched] (5,-1)--(5.65,-1) node[right,text=black,font=\small] {$M$};
\draw[completed] (9,-1)--(9.65,-1) node[right,text=black,font=\small] {$M'$};
\end{tikzpicture}

%% file: figures/octahedron.tex
\begin{tikzpicture}[x=1cm,y=-1cm,line cap=round,line join=round,
  every node/.style={font=\scriptsize},
  octM/.style={draw=Orange,line width=1.1pt},
  octMp/.style={draw=Blue,line width=1.05pt,dash pattern=on 2.6pt off 1.3pt},
  octface/.style={rounded corners=1.5pt,inner xsep=2.2pt,inner ysep=1.6pt},
  octhead/.style={anchor=west,font=\small\bfseries,inner sep=0pt}]
\node[octhead] at (0,0) {(a) Octahedron and shelling $F_1,\ldots,F_8$};
\coordinate (v1) at (2.65,0.65);
\coordinate (v2) at (0.5,4);
\coordinate (v3) at (4.8,4);
\coordinate (v4) at (2.65,3.15);
\coordinate (v5) at (3.2,2.3);
\coordinate (v6) at (2.1,2.3);
\draw[draw=Blue!25,line width=6pt] (v1)--(v2)--(v3)--cycle;
\fill[Orange!19] (v1)--(v2)--(v6)--cycle;
\fill[Green!19] (v1)--(v3)--(v5)--cycle;
\fill[Purple!19] (v1)--(v5)--(v6)--cycle;
\fill[Vermilion!19] (v2)--(v3)--(v4)--cycle;
\fill[SkyBlue!19] (v2)--(v4)--(v6)--cycle;
\fill[Yellow!19] (v3)--(v4)--(v5)--cycle;
\fill[black!19] (v4)--(v5)--(v6)--cycle;
\draw[black!75,line width=.65pt] (v1)--(v2);
\draw[black!75,line width=.65pt] (v1)--(v3);
\draw[black!75,line width=.65pt] (v1)--(v5);
\draw[black!75,line width=.65pt] (v1)--(v6);
\draw[black!75,line width=.65pt] (v2)--(v3);
\draw[black!75,line width=.65pt] (v2)--(v4);
\draw[black!75,line width=.65pt] (v2)--(v6);
\draw[black!75,line width=.65pt] (v3)--(v4);
\draw[black!75,line width=.65pt] (v3)--(v5);
\draw[black!75,line width=.65pt] (v4)--(v5);
\draw[black!75,line width=.65pt] (v4)--(v6);
\draw[black!75,line width=.65pt] (v5)--(v6);
\node[text=Orange!65!black,font=\small] at (1.750,2.317) {$F_2$};
\node[text=Green!65!black,font=\small] at (3.550,2.317) {$F_3$};
\node[text=Purple!65!black,font=\small] at (2.650,1.750) {$F_4$};
\node[text=Vermilion!65!black,font=\small] at (2.650,3.717) {$F_5$};
\node[text=SkyBlue!65!black,font=\small] at (1.750,3.150) {$F_6$};
\node[text=Yellow!65!black,font=\small] at (3.550,3.150) {$F_7$};
\node[text=black!65!black,font=\small] at (2.650,2.583) {$F_8$};
\fill (v1) circle (1.1pt);
\node[above,inner sep=2.5pt] at (v1) {$ 1 $};
\fill (v2) circle (1.1pt);
\node[below left,inner sep=2.5pt] at (v2) {$ 2 $};
\fill (v3) circle (1.1pt);
\node[below right,inner sep=2.5pt] at (v3) {$ 3 $};
\fill (v4) circle (1.1pt);
\node[below,inner sep=2.5pt] at (v4) {$ 4 $};
\fill (v5) circle (1.1pt);
\node[above left,inner sep=2.5pt] at (v5) {$ 5 $};
\fill (v6) circle (1.1pt);
\node[above right,inner sep=2.5pt] at (v6) {$ 6 $};
\node[text=Blue!75!black] at (2.65,4.44) {$F_1=123$ is the outer face};
\node at (6.05,.62) {$i$};
\node at (7.0,.62) {$F_i$};
\node at (8.15,.62) {$R_i$};
\node at (9.55,.62) {$F_i\setminus R_i$};
\node at (10.9,.62) {$r_i$};
\node at (13.25,.62) {$F_i\setminus\{r_i\}\in B_j$};
\draw[black!55,line width=.45pt] (5.7,.91)--(15.1,.91);
\node[circle,inner sep=2pt,fill=Blue!22,draw=Blue!70!black] at (6.05,1.18) {$ 1 $};
\node at (7.0,1.18) {$ 123 $};
\node at (8.15,1.18) {$ \emptyset $};
\node at (9.55,1.18) {$ 123 $};
\node at (10.9,1.18) {$ - $};
\node at (13.25,1.18) {$ P\leftrightarrow F_1 $};
\node[circle,inner sep=2pt,fill=Orange!22,draw=Orange!70!black] at (6.05,1.58) {$ 2 $};
\node at (7.0,1.58) {$ 126 $};
\node at (8.15,1.58) {$ 6 $};
\node at (9.55,1.58) {$ 12 $};
\node at (10.9,1.58) {$ 6 $};
\node at (13.25,1.58) {$ 12\in B_{1} $};
\node[circle,inner sep=2pt,fill=Green!22,draw=Green!70!black] at (6.05,1.98) {$ 3 $};
\node at (7.0,1.98) {$ 135 $};
\node at (8.15,1.98) {$ 5 $};
\node at (9.55,1.98) {$ 13 $};
\node at (10.9,1.98) {$ 5 $};
\node at (13.25,1.98) {$ 13\in B_{1} $};
\node[circle,inner sep=2pt,fill=Purple!22,draw=Purple!70!black] at (6.05,2.38) {$ 4 $};
\node at (7.0,2.38) {$ 156 $};
\node at (8.15,2.38) {$ 56 $};
\node at (9.55,2.38) {$ 1 $};
\node at (10.9,2.38) {$ 5 $};
\node at (13.25,2.38) {$ 16\in B_{2} $};
\node[circle,inner sep=2pt,fill=Vermilion!22,draw=Vermilion!70!black] at (6.05,2.7800000000000002) {$ 5 $};
\node at (7.0,2.7800000000000002) {$ 234 $};
\node at (8.15,2.7800000000000002) {$ 4 $};
\node at (9.55,2.7800000000000002) {$ 23 $};
\node at (10.9,2.7800000000000002) {$ 4 $};
\node at (13.25,2.7800000000000002) {$ 23\in B_{1} $};
\node[circle,inner sep=2pt,fill=SkyBlue!22,draw=SkyBlue!70!black] at (6.05,3.1799999999999997) {$ 6 $};
\node at (7.0,3.1799999999999997) {$ 246 $};
\node at (8.15,3.1799999999999997) {$ 46 $};
\node at (9.55,3.1799999999999997) {$ 2 $};
\node at (10.9,3.1799999999999997) {$ 4 $};
\node at (13.25,3.1799999999999997) {$ 26\in B_{2} $};
\node[circle,inner sep=2pt,fill=Yellow!22,draw=Yellow!70!black] at (6.05,3.58) {$ 7 $};
\node at (7.0,3.58) {$ 345 $};
\node at (8.15,3.58) {$ 45 $};
\node at (9.55,3.58) {$ 3 $};
\node at (10.9,3.58) {$ 4 $};
\node at (13.25,3.58) {$ 35\in B_{3} $};
\node[circle,inner sep=2pt,fill=black!22,draw=black!70!black] at (6.05,3.9800000000000004) {$ 8 $};
\node at (7.0,3.9800000000000004) {$ 456 $};
\node at (8.15,3.9800000000000004) {$ 456 $};
\node at (9.55,3.9800000000000004) {$ \emptyset $};
\node at (10.9,3.9800000000000004) {$ 4 $};
\node at (13.25,3.9800000000000004) {$ 56\in B_{4} $};
\node[anchor=west] at (5.7,4.46) {Color $i$: facet $F_i$, forward block $B_i$, reverse block $B'_i$.};
\node[octhead] at (0,5.05) {(b) Forward cubes connected by the matching $M$};
\draw[octM] (8.3,5.45)--(9.0,5.45) node[right,text=black] {$M$};
\draw[octMp] (10.1,5.45)--(10.8,5.45) node[right,text=black] {$M'$};
\draw[black!30,line width=.35pt] (12.2,5.45)--(12.9,5.45) node[right,text=black] {cube edges};
\node[anchor=west,text=black!65] at (0,5.45) {Ranks at left; unused inter-cube covers omitted.};
\draw[rounded corners=3pt,draw=Blue!65!black,fill=Blue!6,line width=.5pt] (0.48,7.030) rectangle (2.62,11.470);
\draw[rounded corners=3pt,draw=Orange!65!black,fill=Orange!6,line width=.5pt] (3.195,7.030) rectangle (4.605,10.170);
\draw[rounded corners=3pt,draw=Green!65!black,fill=Green!6,line width=.5pt] (5.245,7.030) rectangle (6.655,10.170);
\draw[rounded corners=3pt,draw=Purple!65!black,fill=Purple!6,line width=.5pt] (7.31,7.030) rectangle (7.99,8.870);
\draw[rounded corners=3pt,draw=Vermilion!65!black,fill=Vermilion!6,line width=.5pt] (8.645,7.030) rectangle (10.055,10.170);
\draw[rounded corners=3pt,draw=SkyBlue!65!black,fill=SkyBlue!6,line width=.5pt] (10.71,7.030) rectangle (11.39,8.870);
\draw[rounded corners=3pt,draw=Yellow!65!black,fill=Yellow!6,line width=.5pt] (12.21,7.030) rectangle (12.89,8.870);
\draw[rounded corners=3pt,draw=black!65!black,fill=black!6,line width=.5pt] (13.71,5.730) rectangle (14.39,7.570);
\coordinate (fe) at (1.55,11.200);
\coordinate (f1) at (0.82,9.900);
\coordinate (f2) at (1.55,9.900);
\coordinate (f3) at (2.28,9.900);
\coordinate (f12) at (0.82,8.600);
\coordinate (f13) at (1.55,8.600);
\coordinate (f23) at (2.28,8.600);
\coordinate (f123) at (1.55,7.300);
\coordinate (f6) at (3.9,9.900);
\coordinate (f16) at (3.535,8.600);
\coordinate (f26) at (4.265,8.600);
\coordinate (f126) at (3.9,7.300);
\coordinate (f5) at (5.95,9.900);
\coordinate (f15) at (5.585,8.600);
\coordinate (f35) at (6.315,8.600);
\coordinate (f135) at (5.95,7.300);
\coordinate (f56) at (7.65,8.600);
\coordinate (f156) at (7.65,7.300);
\coordinate (f4) at (9.35,9.900);
\coordinate (f24) at (8.985,8.600);
\coordinate (f34) at (9.715,8.600);
\coordinate (f234) at (9.35,7.300);
\coordinate (f46) at (11.05,8.600);
\coordinate (f246) at (11.05,7.300);
\coordinate (f45) at (12.55,8.600);
\coordinate (f345) at (12.55,7.300);
\coordinate (f456) at (14.05,7.300);
\coordinate (fP) at (14.05,6.000);
\draw[draw=black!30,line width=.35pt] (f1)--(f13);
\draw[draw=black!30,line width=.35pt] (f1)--(fe);
\draw[draw=black!30,line width=.35pt] (f12)--(f123);
\draw[draw=black!30,line width=.35pt] (f12)--(f2);
\draw[draw=black!30,line width=.35pt] (f123)--(f13);
\draw[draw=black!30,line width=.35pt] (f126)--(f16);
\draw[draw=black!30,line width=.35pt] (f135)--(f15);
\draw[draw=black!30,line width=.35pt] (f2)--(f23);
\draw[draw=black!30,line width=.35pt] (f23)--(f3);
\draw[draw=black!30,line width=.35pt] (f234)--(f34);
\draw[draw=black!30,line width=.35pt] (f24)--(f4);
\draw[draw=black!30,line width=.35pt] (f26)--(f6);
\draw[draw=black!30,line width=.35pt] (f3)--(fe);
\draw[draw=black!30,line width=.35pt] (f35)--(f5);
\draw[octMp] (f1)--(f12);
\draw[octMp] (f123)--(f23);
\draw[octMp] (f126)--(f26);
\draw[octMp] (f13)--(f3);
\draw[octMp] (f135)--(f35);
\draw[octMp] (f15)--(f5);
\draw[octMp] (f156)--(f56);
\draw[octMp] (f16)--(f6);
\draw[octMp] (f2)--(fe);
\draw[octMp] (f234)--(f24);
\draw[octMp] (f246)--(f46);
\draw[octMp] (f34)--(f4);
\draw[octMp] (f345)--(f45);
\draw[octMp] (f456)--(fP);
\draw[octM,line width=1.25pt,preaction={draw=white,line width=3.2pt,solid}] (f15) .. controls (5.585,9.350) and (0.820,9.150) .. (f1);
\draw[octM,line width=1.25pt,preaction={draw=white,line width=3.2pt,solid}] (f126) .. controls (3.900,8.050) and (0.820,7.850) .. (f12);
\draw[octM,line width=1.25pt,preaction={draw=white,line width=3.2pt,solid}] (fP) .. controls (14.050,6.750) and (1.550,6.550) .. (f123);
\draw[octM,line width=1.25pt,preaction={draw=white,line width=3.2pt,solid}] (f135) .. controls (5.950,8.050) and (1.550,7.850) .. (f13);
\draw[octM,line width=1.25pt,preaction={draw=white,line width=3.2pt,solid}] (f156) .. controls (7.650,8.050) and (3.535,7.850) .. (f16);
\draw[octM,line width=1.25pt,preaction={draw=white,line width=3.2pt,solid}] (f24) .. controls (8.985,9.350) and (1.550,9.150) .. (f2);
\draw[octM,line width=1.25pt,preaction={draw=white,line width=3.2pt,solid}] (f234) .. controls (9.350,8.050) and (2.280,7.850) .. (f23);
\draw[octM,line width=1.25pt,preaction={draw=white,line width=3.2pt,solid}] (f246) .. controls (11.050,8.050) and (4.265,7.850) .. (f26);
\draw[octM,line width=1.25pt,preaction={draw=white,line width=3.2pt,solid}] (f34) .. controls (9.715,9.350) and (2.280,9.150) .. (f3);
\draw[octM,line width=1.25pt,preaction={draw=white,line width=3.2pt,solid}] (f345) .. controls (12.550,8.050) and (6.315,7.850) .. (f35);
\draw[octM,line width=1.25pt,preaction={draw=white,line width=3.2pt,solid}] (f4) .. controls (9.350,10.650) and (1.550,10.450) .. (fe);
\draw[octM,line width=1.25pt,preaction={draw=white,line width=3.2pt,solid}] (f45) .. controls (12.550,9.350) and (5.950,9.150) .. (f5);
\draw[octM,line width=1.25pt,preaction={draw=white,line width=3.2pt,solid}] (f456) .. controls (14.050,8.050) and (7.650,7.850) .. (f56);
\draw[octM,line width=1.25pt,preaction={draw=white,line width=3.2pt,solid}] (f46) .. controls (11.050,9.350) and (3.900,9.150) .. (f6);
\node[octface,draw=Blue!75!black,fill=Blue!17] (f1) at (0.82,9.900) {$ 1 $};
\node[octface,draw=Blue!75!black,fill=Blue!17] (f12) at (0.82,8.600) {$ 12 $};
\node[octface,draw=Blue!75!black,fill=Blue!17] (f123) at (1.55,7.300) {$ 123 $};
\node[octface,draw=Blue!75!black,fill=Blue!17] (f13) at (1.55,8.600) {$ 13 $};
\node[octface,draw=Blue!75!black,fill=Blue!17] (f2) at (1.55,9.900) {$ 2 $};
\node[octface,draw=Blue!75!black,fill=Blue!17] (f23) at (2.28,8.600) {$ 23 $};
\node[octface,draw=Blue!75!black,fill=Blue!17] (f3) at (2.28,9.900) {$ 3 $};
\node[octface,draw=Blue!75!black,fill=Blue!17] (fe) at (1.55,11.200) {$ \emptyset $};
\node at (1.55,11.750) {$B_{1}$ (dim.\ 3)};
\node[octface,draw=Orange!75!black,fill=Orange!17] (f126) at (3.9,7.300) {$ 126 $};
\node[octface,draw=Orange!75!black,fill=Orange!17] (f16) at (3.535,8.600) {$ 16 $};
\node[octface,draw=Orange!75!black,fill=Orange!17] (f26) at (4.265,8.600) {$ 26 $};
\node[octface,draw=Orange!75!black,fill=Orange!17] (f6) at (3.9,9.900) {$ 6 $};
\node at (3.9,11.750) {$B_{2}$ (dim.\ 2)};
\node[octface,draw=Green!75!black,fill=Green!17] (f135) at (5.95,7.300) {$ 135 $};
\node[octface,draw=Green!75!black,fill=Green!17] (f15) at (5.585,8.600) {$ 15 $};
\node[octface,draw=Green!75!black,fill=Green!17] (f35) at (6.315,8.600) {$ 35 $};
\node[octface,draw=Green!75!black,fill=Green!17] (f5) at (5.95,9.900) {$ 5 $};
\node at (5.95,11.750) {$B_{3}$ (dim.\ 2)};
\node[octface,draw=Purple!75!black,fill=Purple!17] (f156) at (7.65,7.300) {$ 156 $};
\node[octface,draw=Purple!75!black,fill=Purple!17] (f56) at (7.65,8.600) {$ 56 $};
\node at (7.65,11.750) {$B_{4}$ (dim.\ 1)};
\node[octface,draw=Vermilion!75!black,fill=Vermilion!17] (f234) at (9.35,7.300) {$ 234 $};
\node[octface,draw=Vermilion!75!black,fill=Vermilion!17] (f24) at (8.985,8.600) {$ 24 $};
\node[octface,draw=Vermilion!75!black,fill=Vermilion!17] (f34) at (9.715,8.600) {$ 34 $};
\node[octface,draw=Vermilion!75!black,fill=Vermilion!17] (f4) at (9.35,9.900) {$ 4 $};
\node at (9.35,11.750) {$B_{5}$ (dim.\ 2)};
\node[octface,draw=SkyBlue!75!black,fill=SkyBlue!17] (f246) at (11.05,7.300) {$ 246 $};
\node[octface,draw=SkyBlue!75!black,fill=SkyBlue!17] (f46) at (11.05,8.600) {$ 46 $};
\node at (11.05,11.750) {$B_{6}$ (dim.\ 1)};
\node[octface,draw=Yellow!75!black,fill=Yellow!17] (f345) at (12.55,7.300) {$ 345 $};
\node[octface,draw=Yellow!75!black,fill=Yellow!17] (f45) at (12.55,8.600) {$ 45 $};
\node at (12.55,11.750) {$B_{7}$ (dim.\ 1)};
\node[octface,draw=black!75!black,fill=black!17] (f456) at (14.05,7.300) {$ 456 $};
\node[octface,draw=black!75!black,fill=black!17] (fP) at (14.05,6.000) {$ P $};
\node at (14.05,11.750) {$B_{8}$ (dim.\ 1)};
\node[anchor=east,text=black!60] at (-0.12,11.200) {$ 0 $};
\node[anchor=east,text=black!60] at (-0.12,9.900) {$ 1 $};
\node[anchor=east,text=black!60] at (-0.12,8.600) {$ 2 $};
\node[anchor=east,text=black!60] at (-0.12,7.300) {$ 3 $};
\node[anchor=east,text=black!60] at (-0.12,6.000) {$ 4 $};
\node[octhead] at (0,12.250) {(c) Reverse cubes: coordinate matchings and the edge $F_1P$};
\draw[rounded corners=3pt,draw=Blue!65!black,fill=Blue!6,line width=.5pt] (0.51,12.520) rectangle (1.19,13.650);
\draw[rounded corners=3pt,draw=Orange!65!black,fill=Orange!6,line width=.5pt] (2.01,13.110) rectangle (2.69,14.240);
\draw[rounded corners=3pt,draw=Green!65!black,fill=Green!6,line width=.5pt] (3.51,13.110) rectangle (4.19,14.240);
\draw[rounded corners=3pt,draw=Purple!65!black,fill=Purple!6,line width=.5pt] (4.945,13.110) rectangle (6.355,14.830);
\draw[rounded corners=3pt,draw=Vermilion!65!black,fill=Vermilion!6,line width=.5pt] (7.11,13.110) rectangle (7.79,14.240);
\draw[rounded corners=3pt,draw=SkyBlue!65!black,fill=SkyBlue!6,line width=.5pt] (8.545,13.110) rectangle (9.955,14.830);
\draw[rounded corners=3pt,draw=Yellow!65!black,fill=Yellow!6,line width=.5pt] (10.645,13.110) rectangle (12.055,14.830);
\draw[rounded corners=3pt,draw=black!65!black,fill=black!6,line width=.5pt] (12.68,13.110) rectangle (14.82,15.420);
\coordinate (r123) at (0.85,13.380);
\coordinate (rP) at (0.85,12.790);
\coordinate (r12) at (2.35,13.970);
\coordinate (r126) at (2.35,13.380);
\coordinate (r13) at (3.85,13.970);
\coordinate (r135) at (3.85,13.380);
\coordinate (r1) at (5.65,14.560);
\coordinate (r15) at (5.285,13.970);
\coordinate (r16) at (6.015,13.970);
\coordinate (r156) at (5.65,13.380);
\coordinate (r23) at (7.45,13.970);
\coordinate (r234) at (7.45,13.380);
\coordinate (r2) at (9.25,14.560);
\coordinate (r24) at (8.885,13.970);
\coordinate (r26) at (9.615,13.970);
\coordinate (r246) at (9.25,13.380);
\coordinate (r3) at (11.35,14.560);
\coordinate (r34) at (10.985,13.970);
\coordinate (r35) at (11.715,13.970);
\coordinate (r345) at (11.35,13.380);
\coordinate (re) at (13.75,15.150);
\coordinate (r4) at (13.02,14.560);
\coordinate (r5) at (13.75,14.560);
\coordinate (r6) at (14.48,14.560);
\coordinate (r45) at (13.02,13.970);
\coordinate (r46) at (13.75,13.970);
\coordinate (r56) at (14.48,13.970);
\coordinate (r456) at (13.75,13.380);
\draw[draw=black!42,line width=.4pt] (r1)--(r16);
\draw[draw=black!42,line width=.4pt] (r15)--(r156);
\draw[draw=black!42,line width=.4pt] (r2)--(r26);
\draw[draw=black!42,line width=.4pt] (r24)--(r246);
\draw[draw=black!42,line width=.4pt] (r3)--(r35);
\draw[draw=black!42,line width=.4pt] (r34)--(r345);
\draw[draw=black!42,line width=.4pt] (r4)--(r45);
\draw[draw=black!42,line width=.4pt] (r4)--(r46);
\draw[draw=black!42,line width=.4pt] (r45)--(r456);
\draw[draw=black!42,line width=.4pt] (r456)--(r46);
\draw[draw=black!42,line width=.4pt] (r5)--(r56);
\draw[draw=black!42,line width=.4pt] (r5)--(re);
\draw[draw=black!42,line width=.4pt] (r56)--(r6);
\draw[draw=black!42,line width=.4pt] (r6)--(re);
\draw[octM] (r1)--(r15);
\draw[octM] (r12)--(r126);
\draw[octM] (r123)--(rP);
\draw[octM] (r13)--(r135);
\draw[octM] (r156)--(r16);
\draw[octM] (r2)--(r24);
\draw[octM] (r23)--(r234);
\draw[octM] (r246)--(r26);
\draw[octM] (r3)--(r34);
\draw[octM] (r345)--(r35);
\draw[octM] (r4)--(re);
\draw[octM] (r45)--(r5);
\draw[octM] (r456)--(r56);
\draw[octM] (r46)--(r6);
\node[octface,draw=Blue!75!black,fill=Blue!17] (r123) at (0.85,13.380) {$ 123 $};
\node[octface,draw=Blue!75!black,fill=Blue!17] (rP) at (0.85,12.790) {$ P $};
\node at (0.85,15.700) {$B'_{1}$ (dim.\ 1)};
\node[octface,draw=Orange!75!black,fill=Orange!17] (r12) at (2.35,13.970) {$ 12 $};
\node[octface,draw=Orange!75!black,fill=Orange!17] (r126) at (2.35,13.380) {$ 126 $};
\node at (2.35,15.700) {$B'_{2}$ (dim.\ 1)};
\node[octface,draw=Green!75!black,fill=Green!17] (r13) at (3.85,13.970) {$ 13 $};
\node[octface,draw=Green!75!black,fill=Green!17] (r135) at (3.85,13.380) {$ 135 $};
\node at (3.85,15.700) {$B'_{3}$ (dim.\ 1)};
\node[octface,draw=Purple!75!black,fill=Purple!17] (r1) at (5.65,14.560) {$ 1 $};
\node[octface,draw=Purple!75!black,fill=Purple!17] (r15) at (5.285,13.970) {$ 15 $};
\node[octface,draw=Purple!75!black,fill=Purple!17] (r156) at (5.65,13.380) {$ 156 $};
\node[octface,draw=Purple!75!black,fill=Purple!17] (r16) at (6.015,13.970) {$ 16 $};
\node at (5.65,15.700) {$B'_{4}$ (dim.\ 2)};
\node[octface,draw=Vermilion!75!black,fill=Vermilion!17] (r23) at (7.45,13.970) {$ 23 $};
\node[octface,draw=Vermilion!75!black,fill=Vermilion!17] (r234) at (7.45,13.380) {$ 234 $};
\node at (7.45,15.700) {$B'_{5}$ (dim.\ 1)};
\node[octface,draw=SkyBlue!75!black,fill=SkyBlue!17] (r2) at (9.25,14.560) {$ 2 $};
\node[octface,draw=SkyBlue!75!black,fill=SkyBlue!17] (r24) at (8.885,13.970) {$ 24 $};
\node[octface,draw=SkyBlue!75!black,fill=SkyBlue!17] (r246) at (9.25,13.380) {$ 246 $};
\node[octface,draw=SkyBlue!75!black,fill=SkyBlue!17] (r26) at (9.615,13.970) {$ 26 $};
\node at (9.25,15.700) {$B'_{6}$ (dim.\ 2)};
\node[octface,draw=Yellow!75!black,fill=Yellow!17] (r3) at (11.35,14.560) {$ 3 $};
\node[octface,draw=Yellow!75!black,fill=Yellow!17] (r34) at (10.985,13.970) {$ 34 $};
\node[octface,draw=Yellow!75!black,fill=Yellow!17] (r345) at (11.35,13.380) {$ 345 $};
\node[octface,draw=Yellow!75!black,fill=Yellow!17] (r35) at (11.715,13.970) {$ 35 $};
\node at (11.35,15.700) {$B'_{7}$ (dim.\ 2)};
\node[octface,draw=black!75!black,fill=black!17] (r4) at (13.02,14.560) {$ 4 $};
\node[octface,draw=black!75!black,fill=black!17] (r45) at (13.02,13.970) {$ 45 $};
\node[octface,draw=black!75!black,fill=black!17] (r456) at (13.75,13.380) {$ 456 $};
\node[octface,draw=black!75!black,fill=black!17] (r46) at (13.75,13.970) {$ 46 $};
\node[octface,draw=black!75!black,fill=black!17] (r5) at (13.75,14.560) {$ 5 $};
\node[octface,draw=black!75!black,fill=black!17] (r56) at (14.48,13.970) {$ 56 $};
\node[octface,draw=black!75!black,fill=black!17] (r6) at (14.48,14.560) {$ 6 $};
\node[octface,draw=black!75!black,fill=black!17] (re) at (13.75,15.150) {$ \emptyset $};
\node at (13.75,15.700) {$B'_{8}$ (dim.\ 3)};
\node[anchor=east,text=black!60] at (-0.12,15.150) {$ 0 $};
\node[anchor=east,text=black!60] at (-0.12,14.560) {$ 1 $};
\node[anchor=east,text=black!60] at (-0.12,13.970) {$ 2 $};
\node[anchor=east,text=black!60] at (-0.12,13.380) {$ 3 $};
\node[anchor=east,text=black!60] at (-0.12,12.790) {$ 4 $};
\node[octhead] at (0,16.170) {(d) Connectivity: the facet--ridge edges of $M$};
\node[octface,draw=Blue!75!black,fill=Blue!17] (b1) at (1,17.560) {$B_1$};
\node[octface,draw=Orange!75!black,fill=Orange!17] (b2) at (5,16.970) {$B_2$};
\node[octface,draw=Green!75!black,fill=Green!17] (b3) at (5,17.750) {$B_3$};
\node[octface,draw=Vermilion!75!black,fill=Vermilion!17] (b5) at (5,18.530) {$B_5$};
\node[octface,draw=Purple!75!black,fill=Purple!17] (b4) at (9,16.610) {$B_4$};
\node[octface,draw=SkyBlue!75!black,fill=SkyBlue!17] (b6) at (9,17.400) {$B_6$};
\node[octface,draw=Yellow!75!black,fill=Yellow!17] (b7) at (9,18.190) {$B_7$};
\node[octface,draw=black!75!black,fill=black!17] (b8) at (13,16.610) {$B_8$};
\draw[octM,-{Stealth[length=3pt]}] (b2)--node[midway,fill=white,inner sep=1.2pt,text=black] {$ 126\! -\! 12 $} (b1);
\draw[octM,-{Stealth[length=3pt]}] (b3)--node[midway,fill=white,inner sep=1.2pt,text=black] {$ 135\! -\! 13 $} (b1);
\draw[octM,-{Stealth[length=3pt]}] (b4)--node[midway,fill=white,inner sep=1.2pt,text=black] {$ 156\! -\! 16 $} (b2);
\draw[octM,-{Stealth[length=3pt]}] (b5)--node[midway,fill=white,inner sep=1.2pt,text=black] {$ 234\! -\! 23 $} (b1);
\draw[octM,-{Stealth[length=3pt]}] (b6)--node[midway,fill=white,inner sep=1.2pt,text=black] {$ 246\! -\! 26 $} (b2);
\draw[octM,-{Stealth[length=3pt]}] (b7)--node[midway,fill=white,inner sep=1.2pt,text=black] {$ 345\! -\! 35 $} (b3);
\draw[octM,-{Stealth[length=3pt]}] (b8)--node[midway,fill=white,inner sep=1.2pt,text=black] {$ 456\! -\! 56 $} (b4);
\node[anchor=west,align=left] at (10.65,17.810) {Every arrow decreases the index.\\All cubes reach $B_1$.};
\end{tikzpicture}

%% file: bib.bib
@inproceedings{Listingfaces,
  author    = {Nastaran Behrooznia and Sofia Brenner and Arturo Merino and Torsten M\"{u}tze and Christian Rieck and Francesco Verciani},
  title     = {Listing faces of polytopes},
  booktitle = {Proceedings of the 2026 Annual ACM-SIAM Symposium on Discrete Algorithms (SODA)},
  pages     = {6212--6222},
  year      = {2026},
  publisher = {SIAM},
  doi       = {10.1137/1.9781611978971.222}
}

@article{bruggesser1971shellable,
  title={Shellable decompositions of cells and spheres},
  author={Bruggesser, Heinz and Mani, Peter},
  journal={Mathematica Scandinavica},
  volume={29},
  number={2},
  pages={197--205},
  year={1971},
  publisher={Mathematical Societies of the Scandinavian Countries}
}

@article{chari2000discrete,
  author = {Manoj K. Chari},
  title = {On discrete {Morse} functions and combinatorial decompositions},
  journal = {Discrete Mathematics},
  volume = {217},
  number = {1-3},
  pages = {101--113},
  year = {2000},
  doi = {10.1016/S0012-365X(99)00258-7}
}

@article{fink2007perfect,
  title={Perfect matchings extend to Hamilton cycles in hypercubes},
  author={Fink, Ji{\v{r}}{\'\i}},
  journal={Journal of Combinatorial Theory, Series B},
  volume={97},
  number={6},
  pages={1074--1076},
  year={2007},
  publisher={Elsevier}
}

@article{gregor2009perfect,
  title={Perfect matchings extending on subcubes to Hamiltonian cycles of hypercubes},
  author={Gregor, Petr},
  journal={Discrete Mathematics},
  volume={309},
  number={6},
  pages={1711--1713},
  year={2009},
  publisher={Elsevier}
}

@misc{GrayCode,
  author       = {Frank Gray},
  title        = {Pulse Code Communication},
  howpublished = {U.S. Patent 2,632,058},
  month        = {3},
  year         = {1953},
  note         = {Filed Nov. 13, 1947}
}

@book{KnuthBRGC,
  author    = {Donald E. Knuth},
  title     = {The Art of Computer Programming, Volume 4A: Combinatorial Algorithms, Part 1},
  publisher = {Addison-Wesley Professional},
  year      = {2011},
  address   = {Boston, MA, USA},
  isbn      = {978-0201038040}
}

@book{ziegler1995polytopes,
  author    = {G{\"u}nter M. Ziegler},
  title     = {Lectures on Polytopes},
  series    = {Graduate Texts in Mathematics},
  volume    = {152},
  publisher = {Springer},
  location  = {New York},
  year      = {1995},
  doi       = {10.1007/978-1-4613-8431-1}
}

@article{survey_gray_codes,
title = {Combinatorial Gray Codes -- an Updated Survey},
author = {Torsten M{\"u}tze},
Fjournal = {Electronic Journal of Combinatorics},
journal = {Electron. J. Comb.},
year = {2023},
volume = {30},
number = {3},
eid = {DS26},
DOI = {10.37236/11023},
}

@article{Graph_associahedra_general,
   author = {Michael Carr and Satyan L. Devadoss},
   title = {Coxeter complexes and graph-associahedra},
 FJournal = {Topology and its Applications},
 Journal = {Topology Appl.},
   volume = {153},
   number = {12},
   pages = {2155--2168},
   year = {2006},
 DOI = {10.1016/j.topol.2005.08.010},
}

@article{Associahedron,
  author = {Carl W. Lee},
  title = {The Associahedron and Triangulations of the {$n$}-gon},
  journal = {European Journal of Combinatorics},
  volume = {10},
  number = {6},
  pages = {551--560},
  year = {1989},
  doi = {10.1016/S0195-6698(89)80072-1}
}

@article{fink2020algorithms,
  author = {Fink, Ji{\v{r}}{\'\i}},
  title = {Two algorithms extending a perfect matching of the hypercube into a {Hamiltonian} cycle},
  journal = {European Journal of Combinatorics},
  volume = {88},
  eid = {103111},
  year = {2020},
  doi = {10.1016/j.ejc.2020.103111}
}

@incollection{novik2023faces,
  author = {Isabella Novik},
  title = {Face numbers: the upper bound side of the story},
  booktitle = {ICM---International Congress of Mathematicians. Vol. VI. Sections 12--14},
  pages = {4622--4644},
  publisher = {EMS Press},
  location = {Berlin},
  year = {2023},
  doi = {10.4171/ICM2022/43}
}

@incollection{danaraj1978shellable,
  author = {Gopal Danaraj and Victor Klee},
  title = {Which spheres are shellable?},
  booktitle = {Algorithmic Aspects of Combinatorics},
  series = {Annals of Discrete Mathematics},
  volume = {2},
  pages = {33--52},
  publisher = {North-Holland},
  year = {1978},
  doi = {10.1016/S0167-5060(08)70320-0}
}

@article{booleStott1910semiregular,
  author  = {Alicia {Boole Stott}},
  title   = {Geometrical deduction of semiregular from regular polytopes and space fillings},
  journal = {Verhandelingen der Koninklijke Akademie van Wetenschappen te Amsterdam, Eerste Sectie},
  volume  = {11},
  number  = {1},
  pages   = {1--24},
  year    = {1910},
  url     = {https://dwc.knaw.nl/DL/publications/PU00011492.pdf}
}

@article{schoute1913polytopes,
  author  = {Pieter Hendrik Schoute},
  title   = {Analytical treatment of the polytopes regularly derived from the regular polytopes. {Sections II, III, IV}},
  journal = {Verhandelingen der Koninklijke Akademie van Wetenschappen te Amsterdam, Eerste Sectie},
  volume  = {11},
  number  = {5},
  pages   = {1--108},
  year    = {1913},
  url     = {https://dwc.knaw.nl/DL/publications/PU00011496.pdf}
}

@article{feichtner2005matroid,
  author  = {Eva Maria Feichtner and Bernd Sturmfels},
  title   = {Matroid polytopes, nested sets and {Bergman} fans},
  journal = {Portugaliae Mathematica (N.S.)},
  volume  = {62},
  number  = {4},
  pages   = {437--468},
  year    = {2005},
  eprint  = {math/0411260},
  eprinttype = {arxiv}
}

@article{postnikov2009permutohedra,
  author  = {Alexander Postnikov},
  title   = {Permutohedra, associahedra, and beyond},
  journal = {International Mathematics Research Notices},
  volume  = {2009},
  number  = {6},
  pages   = {1026--1106},
  year    = {2009},
  doi     = {10.1093/imrn/rnn153}
}

@article{chapoton2002generalized,
  author = {Fr{\'e}d{\'e}ric Chapoton and Sergey Fomin and Andrei Zelevinsky},
  title = {Polytopal realizations of generalized associahedra},
  journal = {Canadian Mathematical Bulletin},
  volume = {45},
  number = {4},
  pages = {537--566},
  year = {2002},
  doi = {10.4153/CMB-2002-054-1}
}

@article{hohlweg2011generalized,
  author = {Christophe Hohlweg and Carsten Lange and Hugh Thomas},
  title = {Permutahedra and generalized associahedra},
  journal = {Advances in Mathematics},
  volume = {226},
  number = {1},
  pages = {608--640},
  year = {2011},
  doi = {10.1016/j.aim.2010.07.005}
}

@incollection{rote2003pointed,
  author = {G{\"u}nter Rote and Francisco Santos and Ileana Streinu},
  title = {Expansive motions and the polytope of pointed pseudo-triangulations},
  booktitle = {Discrete and Computational Geometry: The Goodman--Pollack Festschrift},
  editor = {Boris Aronov and Saugata Basu and J{\'a}nos Pach and Micha Sharir},
  series = {Algorithms and Combinatorics},
  volume = {25},
  publisher = {Springer},
  pages = {699--736},
  year = {2003},
  doi = {10.1007/978-3-642-55566-4_33}
}

@article{ardila2022bipermutahedron,
  author = {Federico Ardila},
  title = {The bipermutahedron},
  journal = {Combinatorial Theory},
  volume = {2},
  number = {3},
  eid = {1},
  year = {2022},
  doi = {10.5070/C62359149}
}

@article{ardilaDenhamHuh2023lagrangian,
  author = {Federico Ardila and Graham Denham and June Huh},
  title = {Lagrangian geometry of matroids},
  journal = {Journal of the American Mathematical Society},
  volume = {36},
  number = {3},
  pages = {727--794},
  year = {2023},
  doi = {10.1090/jams/1009}
}

@article{pilaudPons2018permutrees,
  author = {Vincent Pilaud and Viviane Pons},
  title = {Permutrees},
  journal = {Algebraic Combinatorics},
  volume = {1},
  number = {2},
  pages = {173--224},
  year = {2018},
  doi = {10.5802/alco.1}
}

@incollection{kleeWitzgall1968transportation,
  author = {Victor Klee and Christoph Witzgall},
  title = {Facets and vertices of transportation polytopes},
  booktitle = {Mathematics of the Decision Sciences, Part I},
  editor = {George B. Dantzig and Veinott, Jr., Arthur F.},
  series = {Lectures in Applied Mathematics},
  volume = {11},
  publisher = {American Mathematical Society},
  location = {Providence, RI},
  pages = {257--282},
  year = {1968}
}

@article{castilloLiu2022nested,
  author = {Federico Castillo and Fu Liu},
  title = {Deformation cones of nested braid fans},
  journal = {International Mathematics Research Notices},
  volume = {2022},
  number = {3},
  pages = {1973--2026},
  year = {2022},
  doi = {10.1093/imrn/rnaa090}
}

@article{palu2021nonkissing,
  author = {Yann Palu and Vincent Pilaud and Pierre-Guy Plamondon},
  title = {Non-kissing complexes and {$\tau$}-tilting for gentle algebras},
  journal = {Memoirs of the American Mathematical Society},
  volume = {274},
  number = {1343},
  year = {2021},
  doi = {10.1090/memo/1343}
}

@misc{dao2019extended,
  author = {Quang Dao and Christina Meng and Julian Wellman and Zixuan Xu and Calvin Yost-Wolff and Teresa Yu},
  title = {Extended nestohedra and their face numbers},
  year = {2019},
  eprint = {1912.00273},
  eprinttype = {arxiv}
}

@misc{pilaud2013signed,
  author = {Vincent Pilaud},
  title = {Signed tree associahedra},
  year = {2013},
  eprint = {1309.5222},
  eprinttype = {arxiv}
}

@phdthesis{almeter2022hypercube,
  author = {Jordan Grady Almeter},
  title = {{P}-graph associahedra and hypercube graph associahedra},
  school = {North Carolina State University},
  year = {2022},
  eprint = {2211.02113},
  eprinttype = {arxiv}
}

@article{orden2005noncrossing,
  author = {David Orden and Francisco Santos},
  title = {The polytope of non-crossing graphs on a planar point set},
  journal = {Discrete \& Computational Geometry},
  volume = {33},
  number = {2},
  pages = {275--305},
  year = {2005},
  doi = {10.1007/s00454-004-1143-1}
}

@article{pilaudStump2015brick,
  author = {Vincent Pilaud and Christian Stump},
  title = {Brick polytopes of spherical subword complexes and generalized associahedra},
  journal = {Advances in Mathematics},
  volume = {276},
  pages = {1--61},
  year = {2015},
  doi = {10.1016/j.aim.2015.02.012}
}

@misc{mantovani2025acyclonestohedra,
  author = {Chiara Mantovani and Arnau Padrol and Vincent Pilaud},
  title = {Facial nested complexes and acyclonestohedra},
  year = {2025},
  eprint = {2509.15914},
  eprinttype = {arxiv}
}
